\documentclass[12 pt]{amsart}

\newtheorem{theorem}{Theorem}[section]
\newtheorem{lemma}[theorem]{Lemma}

\newtheorem{conjecture}[theorem]{Conjecture}
\newtheorem{problem}[theorem]{Problem}

\theoremstyle{definition}

\newtheorem{example}[theorem]{Example}
\newtheorem{remark}[theorem]{Remark}

\usepackage{amscd,amssymb}

\begin{document}

\title[Degree estimate for commutators]
{Degree estimate for commutators}

\author[Vesselin Drensky and Jie-Tai Yu]
{Vesselin Drensky and Jie-Tai Yu}
\address{Institute of Mathematics and Informatics,
Bulgarian Academy of Sciences,
1113 Sofia, Bulgaria}
\email{drensky@math.bas.bg}
\address{Department of Mathematics,
The University of Hong Kong,
Hong Kong SAR, China}
\email{yujt@hkucc.hku.hk}

\thanks
{The research of Vesselin Drensky was partially supported by Grant
MI-1503/2005 of the Bulgarian National Science Fund.}

\thanks
{The research of Jie-Tai Yu was partially supported by an RGC-CERG grant.}

\subjclass[2000]
{16S10; 16W20; 16Z05.}

\begin{abstract} 
Let $K\langle X\rangle$ be a free associative algebra over a field $K$
of characteristic 0 and let each of the noncommuting polynomials 
$f,g\in K\langle X\rangle$ generate its centralizer. 
Assume that the leading homogeneous components  of $f$ and $g$
are algebraically dependent with degrees which do not divide each other. 
We give a counterexample to the recent conjecture of Jie-Tai Yu that 
\[
\text{\rm deg}([f,g])=\text{\rm deg}(fg-gf)
>\min\{\text{\rm deg}(f),\text{\rm deg}(g)\}.
\]
Our example satisfies 
\[
\frac{1}{2}\text{\rm deg}(g)<\text{\rm deg}([f,g])
<\text{\rm deg}(g)<\text{\rm deg}(f)
\]
and  $\deg([f,g])$ can be made as close to $\text{deg}(g)/2$ as we want.
We obtain also a counterexample to another related conjecture 
of Makar-Limanov and Jie-Tai Yu
stated in terms of Malcev --  Neumann formal power series. These counterexamples
are found using the description of the free algebra $K\langle X\rangle$
considered as a bimodule of $K[u]$ where $u$ is a monomial which is not
a power of another monomial and then solving the equation
$[u^m,s]=[u^n,r]$ with unknowns
$r,s\in K\langle X\rangle$.
\end{abstract}

\maketitle

\section*{Introduction}

Let $K$ be a field of characteristic zero and let $X=\{x_1,\ldots,x_d\}$ be
a finite set of variables. Let $K[X]$ and $K\langle X\rangle$ be, respectively, 
the polynomial algebra and the free associative $K$-algebra generated by 
$X$. If $f,g$ are two polynomials in 
$K[X]$ or $K\langle X\rangle$, we want to estimate the minimal degree 
of the elements of the subalgebra generated by them. 
This problem is important
in the study of tame automorphisms of $K[X]$ and $K\langle X\rangle$. 

If $f$ and $g$ are algebraically dependent in $f,g\in K[X]$, 
then the theorem of Zaks \cite{Z}, 
see also Eakin \cite{E} for a simple proof and generalizations, gives that 
their integral closure in $K[X]$ is a polynomial subalgebra $K[h]$. 
If $f$ and $g$ are algebraically dependent
in $K\langle X\rangle$, then they commute, see Cohn \cite{C},
and the theorem of Bergman \cite{B1} 
gives that the centralizer of $f$  is an algebra of the form $K[h]$, 
$h\in K\langle X\rangle$. 
In both the cases not too much is known for the minimal degree of the elements
of the subalgebra generated by $f$ and $g$. 
For example, the famous  Abhyankar -- Moh -- Suzuki
theorem \cite{AM, Su} gives that if $f,g\in K[x]$ 
generate the whole algebra $K[x]$, then
$\text{deg}(f)$ divides $\text{deg}(g)$ or vice versa. 
Also, if $\varphi=(f,g)$ is an
automorphism of $K[x,y]$ or $K\langle x,y\rangle$
(i.e., $\varphi(x)=f$, $\varphi(y)=g$), then
$f$ and $g$ may be of arbitrary high degrees. Then one of the degrees
$\text{deg}(f)$ and $\text{deg}(g)$ divides the other and one of
the leading homogeneous components of $f$ and $g$ is a power of the other.
Clearly $f$ and $g$ generate 
the whole algebra $K[x,y]$ or $K\langle x,y\rangle$. 
Hence there is no useful estimate of the minimal degree of the subalgebra 
generated by $f$ and $g$ if there are no restrictions on their properties.
Several recent results have shown that the natural statement of the
problem is the following:

\begin{problem}\label{estimate the minimal degree}
Let $f$ and $g$ be algebraically independent polynomials 
in $K[X]$ or $K\langle X\rangle$ such that the homogeneous components of
maximal degree of $f$ and $g$ are algebraically dependent. 
If the degrees of $f$ and $g$ do not divide each other, 
find an estimate of the minimal degree of the nonconstant elements 
of the subalgebra generated by $f$ and $g$.
\end{problem}

Using Poisson brackets,
Shestakov and Umirbaev \cite{SU1} gave an estimate for the polynomial case
in terms of the degree of the commutator $[f,g]$ considered as an element
of the free Poisson algebra generated by $X$. This allowed them \cite{SU2}
to discover an algorithm which decides whether an automorphsim of
the polynomial algebra $K[x,y,z]$ is tame and to solve the famous Nagata 
Conjecture \cite{N} that the Nagata automorphism is wild. As a byproduct
of their approach, Shestakov and Umirbaev obtained aslo a new proof of the
Jung -- van der Kulk theorem \cite{J, K} that the automorphisms of $K[x,y]$ are tame.
Later, the estimate was used by Umirbaev and Yu \cite{UY}
to solve a stronger version of the Nagata Conjecture concerning the wildness
of the coordinates of a wild automorphism of $K[x,y,z]$. 

Recently Makar-Limanov and Yu \cite{MLY} have developed a new method 
based on the Lemma on radicals in the Malcev -- Neumann algebra of formal power
series and have obtained an estimate for the minimal degree of the elements
of the subalgebra generated by $f,g$ in $K\langle X\rangle$ depending 
on the degree of the commutator $[f,g]$: If $f$ and $g$ are as in Problem 
\ref{estimate the minimal degree} and $p(x,y)\in K\langle x,y\rangle$, then 
\[
\text{deg}(p(f,g))\geq D(f,g)w_{\text{deg}(f),\text{deg}(g)}(p),
\]
where 
\[
D(f,g)=\frac{\text{deg}([f,g])}{\text{deg}(fg)}
\]
and $w_{\text{deg}(f),\text{deg}(g)}(p)$ is the weighted degree of $p(x,y)$,
defined by
\[
w_{\text{deg}(f),\text{deg}(g)}(x)=\text{deg}(f),\quad
w_{\text{deg}(f),\text{deg}(g)}(y)=\text{deg}(g).
\]
The application of the Lemma on radicals to the commutative case gives
the estimate
\[
\text{deg}(p(f,g))\geq D(f,g)w_{\text{deg}(f),\text{deg}(g)}(p),
\]
where $p(x,y)\in K[x,y]$,
\[
D(f,g)=\left[
1-\frac{(\text{deg}(f),\text{deg}(g))(\text{deg}(fg)-\text{deg}(df\wedge dg))}
{\text{deg}(f)\text{deg}(g)}\right],
\] 
$(m,n)$ is the greatest common divisor of $m,n$ and
\[
df\wedge dg=\sum\left(
\frac{\partial f}{\partial x_i}\frac{\partial g}{\partial x_j}
-\frac{\partial f}{\partial x_j}\frac{\partial g}{\partial x_i}
\right)(dx_i\wedge dx_j)
\]
is the corresponding differential 2-form.

It is easy to see that in principal case (when $p$ has outer rank two),
in noncommutative case we have $\deg(p(f,g))\ge\deg([f,g])$
and in commutative case $\deg(p(f,g))\ge\deg(J(f,g))+2$. See,
for instance, Gong and Yu \cite{GY2}.

These estimates have been used by Jie-Tai Yu \cite{Y1} 
and Gong and Yu \cite{GY1, GY2},  to obtain new results 
on retracts and test elements of $K[x,y]$ and $K\langle x,y\rangle$ as well as
a new proof of the theorem of Czerniakiewicz and Makar-Limanov \cite{Cz, ML}
for the tameness of the automorphisms of $K\langle x,y\rangle$. 

Umirbaev \cite{U1} described the group of tame automorphisms of $K[x,y,z]$
in terms of generators and defining relations. As a consequence, in \cite{U2} 
he developed a method to recognize wild automorphisms of special kind 
of the free algebra $K\langle x,y,z\rangle$. 
In particular, he solved the well known conjecture
that the Anick automorphism of $K\langle x,y,z\rangle$ is wild.
The method of Umirbaev \cite{U2} was further developed by the authors \cite{DY}
in the spirit of the results in \cite{UY}. But up till now, there is no algorithm 
which decides whether a given automorphism of $K\langle x,y,z\rangle$ is tame or wild.
A serious obstacle to the solution of this problem is that there is no estimate
for the degree of the commutator $[f,g]$ for $f,g\in K\langle X\rangle$ being as in
Problem \ref{estimate the minimal degree}.

In his survey \cite{Y2} Jie-Tai Yu raised the following

\begin{conjecture}\label{degree of commutator} 
{\rm (Jie-Tai Yu)}
Let $f$ and $g$ be algebraically independent polynomials 
in $K\langle X\rangle$ such that the homogeneous components of
maximal degree of $f$ and $g$ are algebraically dependent. 
Let $f$ and $g$ generate their centralizers $C(f)$ and $C(g)$ in $K\langle X\rangle$,
respectively. If neither of the degrees of $f$ and $g$ divides the other, then 
\[
\text{\rm deg}([f,g])>\min\{\text{\rm deg}(f),\text{\rm deg}(g)\}.
\]
\end{conjecture}

The condition that the degrees of $f$ and $g$ do not divide each other
is essential. It does not hold when $\varphi=(f,g)$ is an automorphism 
of $K\langle x,y\rangle$ when the commutator test of Dicks \cite{D} gives that 
$[f,g]=\alpha [x,y]$, $0\not=\alpha\in K$. 
The condition that $f$ and $g$ generate their centralizers
is also necessary. For example, if
\[
f=y+(x+y^k)^m,\quad g=(x+y^k)^n,\quad m>n,\quad k>2,
\]
then 
\[
[f,g]=[y,(x+y^k)^n].
\]
The homogeneous component of maximal degree of $[f,g]$ is equal to
\[
[y,xy^{k(n-1)}+y^2xy^{k(n-2)}+\cdots+y^{k(n-1)}x],
\]
\[
\text{deg}([f,g])=k(n-1)+2<kn=\text{deg}(g)<km=\text{deg}(f).
\]

If this conjecture were true, it would give a nice description of the group
of tame automorphisms of $K\langle x,y,z\rangle$, much better than
the description of the group of tame automorphisms of $K[x,y,z]$.
In the approach of Makar-Limanov and Yu \cite{MLY}, they work in the 
Malcev -- Neumann algebra ${\mathcal A}(X)$ of formal power series with monomials from
the free group generated by $X$, allowing infinite sums of homogeneous components of 
negative degree and only finite number of homogeneous components of positive degree.
Conjecture \ref{degree of commutator} would follow from
the following conjecture that was formulated by Makar-Limanov and Jie-Tai Yu
during their attempt to solve Conjecture 0.2.

\begin{conjecture}\label{conjecture on radicals}
{\rm (Makar-Limanov and Jie-Tai Yu)}
Let $g\in K\langle X\rangle$ generate its centralizer and let the homogeneous
component of maximal degree of $g$ is an $n$-th power of an element of $K\langle X\rangle$.
Then, for every $m>n$ which is not divisible by $n$, the formal power series
$g^{m/n}\in {\mathcal A}(X)$ has a monomial of positive degree containing 
a negative power of a generator.
\end{conjecture}

The analogue of Conjecture \ref{degree of commutator} for polynomial algebras is:
{\it If $f$ and $g$ are algebraically independent polynomials 
in $K[X]$ such that the homogeneous components of
maximal degree of $f$ and $g$ are algebraically dependent, 
$f$ and $g$ generate their integral closures $C(f)$ and $C(g)$ in $K[X]$,
respectively, and neither of the degrees of $f$ and $g$ divides the other, then 
\[
\text{\rm deg}(df\wedge dg)>\min\{\text{\rm deg}(f),\text{\rm deg}(g)\}.
\]}

Note that in the case of $K[x,y]$
\[
\text{\rm deg}(df\wedge dg)=\deg(J(f,g))+2
=\text{\rm deg}
\left(\frac{\partial f}{\partial x}\frac{\partial g}{\partial y}
-\frac{\partial f}{\partial y}\frac{\partial g}{\partial x}\right)+2.
\]

Recently, Makar-Limanov has found a simple example of $f,g\in K[x,y]$
such that $f$ and $g$ may be of arbitrary high degree but
\[
\text{\rm deg}(df\wedge dg)=3.
\]
It is easy to see that the analogue of Conjecture
\ref{conjecture on radicals} does not hold in the commutative case.

In the present paper we present a counterexample to Conjectures 
\ref{degree of commutator} and \ref{conjecture on radicals}.
The polynomials $f$ and $g$ in Conjecture \ref{degree of commutator}
are of degree $3(2k+1)$ and $2(2k+1)$, respectively, where
$k\geq 2$, and the degree of the commutator $[f,g]$ is equal to 
$2k+5<\text{deg}(g)<\text{deg}(f)$. The same element $g$ serves as 
a counterexample to Conjecture \ref{conjecture on radicals}.
Comparing with the commutative example of Makar-Limanov, we see that
in the commutative case the quotient
\[
\frac{\text{\rm deg}(df\wedge dg)}{\text{\rm deg}(fg)}
\]
can be very small. In our example, we have that
\[
\frac{\text{deg}([f,g])}{\text{deg}(g)}=\frac{1}{2}+\frac{2}{2k+1}
\]
which can be very close to 1/2. We do not know how far is this quotient from the minimal possible value of the fraction.

\begin{problem}\label{linear estimate of degree}
If $f,g\in K\langle X\rangle$ are as in Conjecture \ref{degree of commutator},
does there exist a positive constant $a$ such that
\[
\text{\rm deg}([f,g])>a\text{\rm deg}(g)?
\]
\end{problem}

If the answer to this problem is affirmative, this still would simplify
the study of the group of tame automorphisms of
$K\langle x,y,z\rangle$.

Although our counterexample is quite simple, in order to find it we have
studied the structure of the free algebra $K\langle X\rangle$ as a bimodule of
$K[u]$, where $u$ is a monomial which is not a proper power. It has turned out that
$K\langle X\rangle$ is a direct sum of three types of bimodules:
the polynomial algebra $K[u]$, free bimodules generated by a single monomial,
and two-generated bimodules with a nontrivial defining relation.
Then we have solved the equation 
\[
[u^m,s]=[u^n,r]
\]
with unknowns $r,s\in K\langle X\rangle$. 
Due to the existence of the $K[u]$-bimodules of the third kind
in $K\langle X\rangle$, we have succeeded to construct the counterexample.

An essential part of the combinatorial theory of free associative algebras is based
on the FIR (free ideal ring) property and the weak Eucledian algorithm
\cite{C}. Also, the theory of equations in $K\langle X\rangle$ may be considered
in the framework of the recently developed universal algebraic geometry,
see the survey by Plotkin \cite{P},
and in the spirit of algebraic geometry over groups, see \cite{BMR, MR}.
Another possibility to consider equations in $K\langle X\rangle$ is from
algorithmic point of view. For example Gupta and Umirbaev \cite{GU}
proved the algorithmic solvability of the problem whether or not a given system  
of linear equations with coefficients in $K\langle X\rangle$ is consistent.
But very little is known about the concrete form of the solutions of an explicitly
given equation. For example, recently Remeslennikov and St\"ohr \cite{RS}
studied the equation $[x,a]+[y,b]=0$ with unknowns $x,y$ in the free Lie algebra $L(X)$ 
where $a,b$ are free generators of $L(X)$.
Hence the description of $K\langle X\rangle$ as a $K[u]$-bimodule 
and the solution of the equation $[u^m,s]=[u^n,r]$
are naturally related with the combinatorial and algorithmic properties of 
free algebras.

\section{The example of Makar-Limanov}

In this section we present the example of Makar-Limanov of two polynomials
$f,g\in K[x,y]$ such that the degrees of $f$ and $g$ are arbitrary high,
do not divide each other and the degree of the Jacobian of $f$ and $g$ is equal to 1,
that answered the commutative version of Conjecture 0.3 
negatively.
It shows that it is unlikely to solve the famous Jacobian conjecture
by degree estimate, as suggested in \cite{SU1}.
The example was communicated by Makar-Limanov to Jie-Tai Yu
in August 2007 when they were trying to attack Conjecture 0.2 and Conjecture 0.3.

\begin{example}\label{example of ML}
Let $a>b$ be positive integers such that $a-b>1$ and $a-b$ divides $a+1$. Let 
\[
c=a-b,\quad k=\frac{a+1}{c},
\]
\[
f(x,y)=yp(x^ay^b),\quad g(x,y)=xy(1+x^ay^b),
\]
where $p(z)\in K[z]$ is a polynomial of degree $k$.
Then 
\[
\text{\rm deg}(f)=(a+b)k+1=(a+b)\frac{a+1}{a-b}+1=\frac{(a+b+2)a}{a-b}
=\frac{(a+b+2)a}{c},
\]
\[
\text{\rm deg}(g)=a+b+2<(a+b+2)\frac{a}{a-b}=\text{\rm deg}(f).
\]
Clearly, $a-b$ cannot divide $a$ because divides $a+1$ and $a-b>1$. Also,
the degree of $f$ and $g$ can be made as large as we want.
Since $f$ cannot be presented in the form $q(h)$ for a polynomial $h$
of lower degree, it generates its integral closure in $K[x,y]$,
and similarly for $g$.

Direct computations show that
\[
J(f,g)=\frac{\partial f}{\partial x}\frac{\partial g}{\partial y}
-\frac{\partial f}{\partial y}\frac{\partial g}{\partial x}
\]
\[
=y\left[-(1+(a+1)x^ay^b)p(x^ay^b)+(a-b)x^ay^b(1+x^ay^b)p'(x^ay^b)\right],
\]
where $p'(z)$ is the derivative of $p(z)$. We want to choose $p(z)$ in such a way that
$J(f,g)=y$. This is equivalent to the condition
\[
-(1+kcz)p(z)+cz(1+z)p'(z)=1.
\]
Let
\[
p(z)=-1+p_1z+p_2z^2+\ldots+p_{k-1}z^{k-1}+p_kz^k,\quad p_i\in K.
\]
Then 
\[
1=-(1+kcz)p(z)+cz(1+z)p'(z)
\]
\[
=-(1+kcz)(-1+p_1z+p_2z^2+\ldots+p_kz^k)+cz(1+z)(p_1+2p_2z+\ldots+kp_kz^{k-1})
\]
\[
=1-(p_1-kc)z-(p_2+kcp_1)z^2-\cdots-(p_k+kcp_{k-1})z^k-kcp_kz^{k+1}
\]
\[
+cp_1z+c(2p_2+p_1)z^2+\cdots+c(kp_k+(k-1)p_{k-1})z^k+kcp_kz^{k+1}
\]
\[
=1+((c-1)p_1+kc)z+((2c-1)p_2+c(1-k)p_1)z^2
+((3c-1)p_3+c(2-k)p_2)z^3
\]
\[
+\cdots+((kc-1)p_k-cp_{k-1})z^k.
\]
Hence
\[
p_1=-\frac{kc}{c-1},\quad p_2=\frac{(k-1)c}{2c-1}p_1,\quad
p_3=\frac{(k-2)c}{3c-1}p_2,\quad\cdots,\quad p_k=\frac{c}{kc-1}p_{k-1}.
\]
Since $p_1\not=0$, we obtain that $p_i\not=0$ for all $i$. Hence
the degree of $f$ is really equal to the prescribed
\[
\text{\rm deg}(f)=\frac{(a+b+2)a}{c}
\]
and 
\[
\text{\rm deg}(J(f,g))=\text{\rm deg}(y)=1.
\]
\end{example}

\section{The free algebra as a bimodule}

Let $\langle X\rangle$ be the free semigroup generated by $X$.
In this section we fix a monomial $u\in \langle X\rangle$ of positive degree
which is not 
a proper power of another monomial. We consider the algebra $K\langle X\rangle$ 
as a $K[u]$-bimodule. Equivalently,
$K\langle X\rangle$ is a $K[u_1,u_2]$-module with action of $u_1$ and $u_2$
defined by
\[
u_1w=uw,\quad u_2w=wu,\quad w\in \langle X\rangle.
\]
Clearly, $K\langle X\rangle$ decomposes as a $K[u_1,u_2]$-module as
\[
K[u]\bigoplus\left(\sum K[u_1,u_2]v\right),
\]
where the inner sum runs on all $v\in \langle X\rangle$ which do not commute with $u$.
We want to find the complete description   
of the $K[u_1,u_2]$-module $K\langle X\rangle$.
If $v_1,v_2\in\langle X\rangle$, we call $v_1$ a beginning, respectively a tail
of $v_2$ if there exists $w\in\langle X\rangle$ such that $v_2=v_1w$,
respectively $v_2=wv_1$.

\begin{theorem}\label{relations as bimodule}
As a $K[u_1,u_2]$-module, $K\langle X\rangle$ 
is a direct sum of three types of submodules:
{\rm (i)} $K[u]$; {\rm (ii)} $K[u_1,u_2]t$; 
{\rm (iii)} $K[u_1,u_2]t_1+K[u_1,u_2]t_2$, where:

{\rm (i)} $K[u]$ is generated as a $K[u_1,u_2]$-module by $1$, and $u^p=u_1^p\cdot 1$.
The defining relation for this submodule is $u_1\cdot 1=u_2\cdot 1$;

{\rm (ii)} $K[u_1,u_2]t$ is a free $K[u_1,u_2]$-module and 
$u$ is neither a beginning nor
a tail of $t$. If $t$ is a beginning, respectively a tail of $u$, 
and $t'$ is the tail, respectively the beginning of $u$ of the same degree as $t$,
then $tu\not=ut'$, respectively $ut\not=t'u$;

{\rm (iii)} $t_1$ and $t_2$ are of the same degree and are, respectively,
a proper beginning and a proper tail of $u$ such that $t_1u=ut_2$. 
The defining relation of this
submodule is $u_2t_1=u_1t_2$. There exist $v_1,v_2\in\langle X\rangle$ with
$v_1v_2\not=v_2v_1$ and a positive integer $k$ such that 
\[
u=(v_1v_2)^kv_1,\quad t_1=v_1v_2,\quad t_2=v_2v_1.
\]
 \end{theorem}

\begin{proof}
The statement (i) is obvious so we only concentrate on (ii) and (iii). 
Each $v\in\langle X\rangle$ has the form $v=u^av'$, where $u$ is not a beginning of $v'$.
Similarly, $v'=tu^b$, where $u$ is not a tail of $t$. Hence, by the property that $u$
is not a proper power of another monomial, we conclude that $K\langle X\rangle$ 
is generated as a $K[u_1,u_2]$-module by 1 and monomials $t$ which do not commute
with $u$ and $u$ is neither a beginning nor a tail of $t$. Let 
\begin{equation}\label{explicit relation}
\sum_{i=1}^p\gamma_iu^{a_i}t_iu^{b_i}=0,\quad 0\not=\gamma_i\in K,
\end{equation}
be a relation between such $t_i$, where the triples $(a_i,b_i,t_i)$
are pairwise different, with possible $t_i=t_j$ for some $i\not=j$.
We may assume that this relation is homogeneous, i.e.,
$(a_i+b_i)\text{deg}(u)+\text{deg}(t_i)$ is the same for all monomials. 
For each $t_i$ there exists a $t_j$ such that $u^{a_i}t_iu^{b_i}=u^{a_j}t_ju^{b_j}$.
Let $u^{a_1}t_1u^{b_1}=u^{a_2}t_2u^{b_2}$.
We may assume that $a_1\leq a_2$.
We cancel $u^{a_1}$ and obtain
that $t_1u^{b_1}=u^at_2u^{b_2}$, $a=a_2-a_1$. 
Similarly, if $b_1\leq b_2$, then $t_1=u^at_2u^b$, $b=b_2-b_1$.
By the choice of $t_1,t_2$ we derive that $t_1=t_2$, $a=b=0$, which contradicts with
$(a_1,b_1,t_1)\not=(a_2,b_2,t_2)$.
If $b_1>b_2$, then for $b=b_1-b_2$
\[
t_1u^b=u^at_2, \quad a,b>0.
\] 
If $\text{deg}(t_1)\geq \text{deg}(u)$, then $u$ is a beginning of $t_1$
which is impossible. Hence $t_1$ is a beginning of $u$.
Similarly if $\text{deg}(t_2)\geq \text{deg}(u)$,
then $u$ is a tail of $t_2$ which is also impossible. Hence $t_2$ is a tail of $u$.
In this way, in the relation (\ref{explicit relation}) all $t_i$ are beginnings
or tails of $u$. Since (\ref{explicit relation}) is homogeneous, the degree of $t_i$ is 
equal to the residue of the division of the degree of the relation by the degree of $u$.
Hence all $t_i$ are of the same degree smaller than the degree of $u$.
Since the beginnings and the tails of $u$ are determined by their degrees,
we obtain that in (\ref{explicit relation}) $p=2$ and $t_1$ is a beginning and $t_2$ is
a tail of $u$. Let $u=t_1w_1=w_2t_2$, $w_1,w_2\in\langle X\rangle$. 
Since $\text{deg}(t_1)=\text{deg}(t_2)<\text{deg}(u)$, 
\[
\text{deg}(w_1)=\text{deg}(u)-\text{deg}(t_1)
=\text{deg}(u)-\text{deg}(t_2)=\text{deg}(w_2).
\]
Replacing $u=t_1w_1=w_2t_2$
in $t_1u^b=u^at_2$, we obtain
\[
t_1u^b=t_1(w_2t_2)\cdots (w_2t_2)=(t_1w_1)\cdots (t_1w_1)t_2=u^at_2.
\]
Both sides of this equality start with $t_1w_2$ and $t_1w_1$, respectively.
Since $w_1$ and $w_2$ are of the same degree, this implies that $w_1=w_2=w$ and
$u=t_1w=wt_2$. Hence 
\[
t_1u=t_1(wt_2)=(t_1w)t_2=ut_2.
\]
If $t_1=t_2$, then $t_1u=ut_1$ which is impossible because $u$ 
is not a proper power and generates its centralizer.
Hence $t_1\not=t_2$. Using the relation $t_1u=ut_2$ we present the elements of
$K[u_1,u_2]t_1+K[u_1,u_2]t_2$ as linear combinations of
$u_1^at_1=u^at_1$ and $u_1^bu_2^ct_2=u^bt_2u^c$.
It is easy to see that $u^{b_1}t_2u^{c_1}\not= u^{b_2}t_2u^{c_2}$,
because $t_2u\not=ut_2$ and $(b_1,c_1)\not=(b_2,c_2)$.
Similarly $u^at_1=u^bt_2u^c$ is also impossible, because 
$t_1$ is not a tail of $u$ (hence $c=0$) and
$\text{deg}(t_1)=\text{deg}(t_2)$,
$t_1\not=t_2$. 
Hence all relations in the $K[u_1,u_2]$-module generated by $t_1$ and $t_2$ follow from $t_1u=ut_2$. Let $u=t_1^kv_1$, where $k$ is maximal with this property. Then $t_1u=ut_2$
implies that
\[
t_1^{k+1}v_1=t_1^kv_1t_2,\quad t_1v_1=v_1t_2.
\]
Since $t_1$ is not a beginning of $v_1$ (otherwise $u=t_1^{k+1}v_1'$),
we obtain that $v_1$ is a beginning of $t_1$ and $t_1=v_1v_2$ for some 
$v_2\in \langle X\rangle$. Now $t_1v_1=v_1t_2$ gives $v_1v_2v_1=v_1t_2$
and $t_2=v_2v_1$. Hence
\[
u=(v_1v_2)^kv_1,\quad t_1=v_1v_2,\quad t_2=v_2v_1
\]
and $v_1v_2\not=v_2v_1$ because $u$ is not a proper power.
\end{proof}

\begin{remark}\label{several t1 and t2 for one u}
For a fixed $u\in\langle X\rangle$ there may be several
pairs $(t_1,t_2)$ satisfying the condition (iii) of Theorem \ref{relations as bimodule}
but all of them have to be of different degree.
For example, if $u=(xy)^kx$, $k>1$, then for any positive $\ell\leq k$ 
the monomials $t_{1\ell}=(xy)^{\ell}$, $t_{2\ell}=(yx)^{\ell}$ satisfy
$t_{1\ell}u=ut_{2\ell}$.
\end{remark}

Now we are going to  solve the equation $[u^m,s]=[u^n,r]$. It is more convenient
to replace $m$ and $n$ with $\ell m$ and $\ell n$, respectively, where $m$ and $n$ are 
coprime.

\begin{example}\label{solution of commutator equation}
Let $u\in \langle X\rangle$ be a monomial of positive degree which is not a power
of another polynomial. Let $\ell, m,n$ be positive integers such that 
$m>n$ and $m,n$ are
coprime. We consider the equation
\begin{equation}\label{commutator equation}
[u^{\ell m},s]=[u^{\ell n},r].
\end{equation}
Applying Theorem \ref{relations as bimodule}, we write $r$ and $s$ in the form
\[
r=r_1(u)+\sum pt+\sum (p_1t_1+p_2t_2),
\quad r_1\in K[u],\quad p,p_1,p_2\in K[u_1,u_2],
\]
\[
s=s_1(u)+\sum qt+\sum (q_1t_1+q_2t_2),
\quad s_1\in K[u],\quad q,q_1,q_2\in K[u_1,u_2],
\]
where the sums run, respectively, on all monomials $t$ and $t_1,t_2$
described in parts (ii) and (iii) of Theorem \ref{relations as bimodule}.
Clearly, $r_1(u)$ and $s_1(u)$ may be arbitrary polynomials and we have to solve
the following systems for each $t$ and $t_1,t_2$:
\begin{equation}\label{system for t}
[u^{\ell m},q(u_1,u_2)t]=[u^{\ell n},p(u_1,u_2)t],
\end{equation}
\begin{equation}\label{system for t1 and t2}
[u^{\ell m},q_1t_1+q_2t_2]=[u^{\ell n},p_1t_1+p_2t_2].
\end{equation}
We rewrite (\ref{system for t}) in the form
\[
(u_1^{\ell m}-u_2^{\ell m})q(u_1,u_2)=(u_1^{\ell n}-u_2^{\ell n})p(u_1,u_2).
\]
Since $m$ and $n$ are coprime, the greatest common divisor of 
the polynomials $u_1^{\ell m}-u_2^{\ell m}$ and
$u_1^{\ell n}-u_2^{\ell n}$ is equal to $u_1^{\ell}-u_2^{\ell}$ and
we obtain that
\[
p(u_1,u_2)=\frac{u_1^{\ell m}-u_2^{\ell m}}{u_1^{\ell}-u_2^{\ell}}r_2(u_1,u_2),
\quad
q(u_1,u_2)=\frac{u_1^{\ell n}-u_2^{\ell n}}{u_1^{\ell}-u_2^{\ell}}r_2(u_1,u_2),
\]
where $r_2(u_1,u_2)\in K[u_1,u_2]$ is an arbitrary polynomial.

Now we assume that 
$\text{deg}(t_1)=\text{deg}(t_2)<\text{deg}(u)$ 
and $u,t_1,t_2$ satisfy the condition $t_1u=ut_2$. Using this relation
we present $p_1t_1+p_2t_2$ and $q_1t_1+q_2t_2$ in (\ref{system for t1 and t2}) 
in the form
\[
p_1t_1+p_2t_2=p_1(u_1)t_1+p_2(u_1,u_2)t_2,\quad
q_1t_1+q_2t_2=q_1(u_1)t_1+q_2(u_1,u_2)t_2
\]
and rewrite (\ref{system for t1 and t2}) as
\[
(u_1^{\ell m}-u_2^{\ell m})(q_1(u_1)t_1+q_2(u_1,u_2)t_2)
=(u_1^{\ell n}-u_2^{\ell n})(p_1(u_1)t_1+p_2(u_1,u_2)t_2).
\]
We replace $u_2t_1$ with $u_1t_2$ and obtain
\[
u_1^{\ell m}q_1t_1-u_1u_2^{\ell m-1}q_1t_2
+(u_1^{\ell m}-u_2^{\ell m})q_2t_2
\]
\[
=u_1^{\ell n}p_1t_1-u_1u_2^{\ell n-1}p_1t_2
+(u_1^{\ell n}-u_2^{\ell n})p_2t_2.
\]
Comparing the coefficients of $t_1$ and $t_2$, we derive
\[
u_1^{\ell m}q_1(u_1)=u_1^{\ell n}p_1(u_1),
\]
\[
-u_1u_2^{\ell m-1}q_1
+(u_1^{\ell m}-u_2^{\ell m})q_2
=-u_1u_2^{\ell n-1}p_1
+(u_1^{\ell n}-u_2^{\ell n})p_2.
\]
It is sufficient to solve these equations when $p_i,q_i$ are homogeneous.
We may assume that $\text{deg}(q_1)=\text{deg}(q_2)=a$,
$\text{deg}(p_1)=\text{deg}(p_2)=a+\ell(m-n)$. Hence
\[
p_1(u_1)=\xi u_1^{a+\ell (m-n)},\quad 
q_1(u_1)=\xi u_1^a,\quad \xi\in K,
\]
\[
-\xi u_1^{a+1}u_2^{\ell m-1}
+(u_1^{\ell m}-u_2^{\ell m})q_2
=-\xi u_1^{a+\ell (m-n)+1}u_2^{\ell n-1}
+(u_1^{\ell n}-u_2^{\ell n})p_2,
\]
\[
(u_1^{\ell m}-u_2^{\ell m})q_2
+\xi u_1^{a+1}u_2^{\ell n-1}(u_1^{\ell (m-n)}-u_2^{\ell (m-n)})
=(u_1^{\ell n}-u_2^{\ell n})p_2.
\]
Defining the polynomial
\[
\Phi_b(u_1,u_2)=\frac{u_1^{\ell b}-u_2^{\ell b}}{u_1^{\ell}-u_2^{\ell}}
=u_1^{\ell (b-1)}+u_1^{\ell (b-2)}u_2^{\ell}+\cdots+u_2^{\ell (b-1)},
\quad b\geq 1,
\]
and using that
\[
\Phi_m(u_1,u_2)=u_1^{\ell (m-n)}\Phi_n(u_1,u_2)+u_2^{\ell n}\Phi_{m-n}(u_1,u_2),
\]
the equation for $\xi,p_2,q_2$ becomes
\[
(u_1^{\ell (m-n)}\Phi_n+u_2^{\ell n}\Phi_{m-n})q_2
+\xi u_1^{a+1}u_2^{\ell n-1}\Phi_{m-n}
=\Phi_np_2,
\]
\[
(p_2-u_1^{\ell (m-n)}q_2)\Phi_n=
u_2^{\ell n-1}(u_2q_2+\xi u_1^{a+1})\Phi_{m-n}.
\]
Since the polynomials $\Phi_n(u_1,u_2)$ and $u_2^{\ell n-1}\Phi_{m-n}(u_1,u_2)$
are coprime, we obtain
\[
p_2-u_1^{\ell (m-n)}q_2=u_2^{\ell n-1}\Phi_{m-n}r_3,\quad
u_2q_2+\xi u_1^{a+1}=\Phi_nr_3,
\]
where $r_3(u_1,u_2)\in K[u_1,u_2]$.
Hence it is sufficient to solve the equation
\[
u_2q_2(u_1,u_2)+\xi u_1^{a+1}=\Phi_n(u_1,u_2)r_3(u_1,u_2)
\]
for $\xi\in K$ and for homogeneous 
$q_2,r_3\in K[u_1,u_2]$.
Comparing the coefficients of $u_1^{a+1}$
and using that $\Phi_n=u_1^{\ell(n-1)}+u_2^{\ell}\Phi_{n-1}$, we obtain 
\[
a+1\geq\text{deg}(\Phi_n)=\ell(n-1),
\]
\[
r_3(u_1,u_2)=\xi u_1^{a-\ell (n-1)+1}+u_2s_3(u_1,u_2),
\quad s_3\in K[u_1,u_2],
\]
\[
q_2(u_1,u_2)=(u_1^{\ell(n-1)}+u_2^{\ell}\Phi_{n-1})s_3
+\xi u_1^{a-\ell(n-1)+1}u_2^{\ell-1}\Phi_{n-1},
\]
for any $\xi\in K$ and arbitrary homogeneous polynomial 
$s_3(u_1,u_2)\in K[u_1,u_2]$ of degree $a-\ell(n-1)$.
\end{example}

It is naturally to ask whether the structure of $K\langle X\rangle$ 
considered as a bimodule of $K[f]$, 
when $f\in K\langle X\rangle$ is an arbitrary polynomial,
is similar to that in Theorem \ref{relations as bimodule}.
The following example shows that in this case some phenomena appear
similar to those in the Buchberger algorithm for the Gr\"obner basis of an ideal.
We do not expect a nice bimodule structure of $K\langle X\rangle$ in the general case.

\begin{example}\label{bimodule structure for any polynomial}
Let us order the monomials of $\langle x,y\rangle$ first by degree and then
lexicographically, assuming that $x>y$. Let
\[
f=xyx+yxx, \quad u=xyx,\quad t_1=xy,\quad t_2=yx.
\]
The leading monomial of $f$ is $u$ and we have $t_1u=ut_2$.
Direct computation gives that
\[
ft_1-ft_2+t_2f=(xy+yx)yxx
\]
belongs to the $K[f]$-bimodule generated by $t_1$ and $t_2$
but its leading monomial $xyyxx$ neither starts or ends with $u$.
\end{example}

\section{The counterexample to Conjecture \ref{degree of commutator}}

The following result presents a counterexample 
to Conjecture \ref{degree of commutator}.

\begin{theorem}\label{the counterexample to first conjecture}
Let $X=\{x,y\}$, $k\geq 2$, and let 
\[
u=(xy)^kx,\quad v=xy,\quad w=yx,
\]
\[
f=u^3+r,\quad r=uv+uw+wu,
\]
\[
g=u^2+s,\quad s=v+w.
\]
Then $f$ and $g$ are algebraically independent polynomials 
 which generate their centralizers $C(f)$ and $C(g)$ in $K\langle x,y\rangle$.
The homogeneous components of maximal degree of $f$ and $g$ are algebraically dependent 
and neither of the degrees of $f$ and $g$ divides the other. Then 
\[
\text{\rm deg}([f,g])<\text{\rm deg}(g)<\text{\rm deg}(f).
\]
The quotient
\[
\frac{\text{deg}([f,g])}{\text{deg}(g)}
=\frac{1}{2}+\frac{2}{2k+1}
\]
is bigger that $1/2$ but can be made as close to $1/2$ as we want by increasing $k$.
\end{theorem}

\begin{proof}
Since $[f,g]\not=0$, 
we derive that $f$ and $g$ are algebraically independent.
The homogeneous components of maximal degree of
$f$ and $g$ are $u^3$ and $u^2$, respectively, and are algebraically dependent.
Their degrees $6k+3$ and $4k+2$ do not divide each other. Direct computations give
\[
[f,g]=[u^3,s]+[r,u^2]+[r,s].
\]
Since, as in Theorem \ref{relations as bimodule} (iii) 
\[
vu=(xy)(xy)^kx=(xyx)^k(yx)=uw, 
\]
we obtain that
\[
[u^3,s]=[u^3,v+w]=u^3(v+w)-(v+w)u^3
\]
\[
=u^3v+u^3w-vu^3-wu^3
=u^3v+u^3w-uwu^2-wu^3,
\]
\[
[r,u^2]=(uv+uw+wu)u^2-u^2(uv+uw+wu)
\]
\[
=uvu^2+uwu^2+wu^3-u^3v-u^3w-u^2wu
\]
\[
=u^2wu+uwu^2+wu^3-u^3v-u^3w-u^2wu=uwu^2+wu^3-u^3v-u^3w,
\]
\[
[u^3,s]+[r,u^2]=0.
\]
Hence,
\[
[f,g]=[r,s]=uvv+uwv+wuv+uvw+uww+wuw
\]
\[
-vuv-vuw-vwu-wuv-wuw-wwu
\]
\[
=uvv+uwv+uvw+uww
-vuv-vuw-vwu-wwu
\]
\[
=uvv+uwv+uvw+uww
-uwv-uww-vwu-wwu
\]
\[
=uvv+uvw-vwu-wwu
\]
\[
=(xy)^kx(xy)(xy)+(xy)^kx(xy)(yx)-(xy)(yx)(xy)^kx-(yx)(yx)(xy)^kx\not=0,
\]
\[\text{deg}([f,g])=\text{deg}(r)+\text{deg}(s)
=4+\text{deg}(u)=2k+5<4k+2=\text{deg}(g).
\]
Clearly, for $k$ sufficiently large, we can make the quotion of the degrees of
$[f,g]$ and $g$ as close to 1/2 as we want.
\end{proof}

The counterexample of Theorem \ref{the counterexample to first conjecture}
has been found applying Theorem \ref{relations as bimodule}
and Example \ref{solution of commutator equation}. We start with
\[
u=(xy)^kx,\quad v=t_1=xy,\quad w=t_2=yx.
\]
Hence $vu=uw$. Working in the $K[u]$-subbimodule of $K\langle x,y\rangle$
generated by $t_1$ and $t_2$, we 
search for $r$ and $s$ such that $[u^m,s]=[u^n,r]$. Then for
\[
f=u^m+r,\quad g=u^n+s, 
\]
where $m>n>1$ and $n$ does not divide $m$, we obtain
\[
[f,g]=[u^m,s]+[r,u^n]+[r,s]=[r,s].
\]
By Example \ref{solution of commutator equation} 
the equation $[u^m,s]=[u^n,r]$ has a partial solution
\[
r=ut_1+ut_2+t_2u,\quad s=t_1+t_2
\]
obtained for
\[
m=3,n=2,\quad p_1=u_1,p_2=u_1+u_2,\quad q_1=q_2=1.
\]
But this approach does not allow to find a solution with  
$\text{deg}([f,g])\leq\text{deg}(g)/2$.

Trying to decrease the degree of $[f,g]$ further, as in the example
of Makar-Limanov, we may add new homogeneous summands to $f$, e.g.
\[
f=u^m+r+r_1,\quad \text{deg}(r_1)<\text{deg}(r),
\] 
such that $[r,s]+[r_1,u^n]=0$. But we face computational (and maybe principal)
difficulties: The monomials of $[r,s]$ are of the form
$u^at_iu^bt_ju^c$, $t_i,t_j=v,w$. Using the relation $vu=uw$, we may assume
that $b=0$ if $t_i=v$ or $t_j=w$. Hence 
\[
[r,s]=\sum h_bwu^bv+h_{11}vv+h_{12}vw+h_{22}ww,\quad h_b,h_{ij}\in K[u_1,u_2].
\]
Since the monomials $wu^bv,vv,vw,ww$ are neither beginnings nor tails of $u$,
we have to work in a free $K[u]$-bimodule and do not know how to
find $r,s,r_1$ of sufficiently small degree such that $[f,g]=[r_1,s]$ and
$\text{deg}([f,g])\leq\text{deg}(g)/2$. The computations become even worst if
we try to add one more component to $g$:
\[
f=u^m+r+r_1,\quad \text{deg}(r_1)<\text{deg}(r),\quad
g=u^n+s+s_1,\quad \text{deg}(s_1)<\text{deg}(s).
\]

\section{Working in the Malcev -- Neumann algebra}

Let $F(X)$ be the free group generated by $X$.
We define the total degree of $u=x_{i_1}^{\pm 1}\cdots x_{i_k}^{\pm 1}\in F(X)$
in the usual way, assuming that $\text{deg}(x_i^{\pm 1})=\pm 1$.
By the theorem of Neumann -- Shimbireva \cite{N1, S}, the group $F(X)$ can be
ordered linearly in many ways. In particular, see Theorem 2.3 in \cite{N1},
if $H$ is a linearly ordered factor group of $F(X)$, then the order of $H$ can be
lifted to a linear order of $F(X)$. Defining a partial order on the free abelian group
generated by $X$ by total degree and then refining it in an arbitrary way, e.g.
lexicographically, we obtain a linear order on $F(X)$ such that 
if $\text{deg}(u_1)<\text{deg}(u_2)$, then $u_1<u_2$. 
Since $\langle X\rangle\subset F(X)$,
we assume that the elements of $\langle X\rangle$ 
are linearly ordered in the same way. If
\[
g=g(X)=\sum_{i=1}^p\alpha_iu_i,\quad 0\not=\alpha_i\in K, 
\quad u_i\in \langle X\rangle,
\quad u_1>u_2>\cdots>u_p,
\]
we denote by $\nu(g)$ the leading monomial $\alpha_1u_1$ of $g$. 
We denote by ${\mathcal A}(X)$ the Malcev -- Neumann algebra of formal power series
used by Malcev and Neumann \cite{M, N2} to show that the group algebra of
an ordered group can be embedded into a division ring. The algebra ${\mathcal A}(X)$
consists of all formal sums 
\[
\tau=\sum_{u\in \Delta}\alpha_uu,\quad \alpha_u\in K,
\]  
where $\Delta$ is a well ordered subset of $F(X)$. 
(For commutative objects this construction was used by Hahn \cite{H}.)
We shall use ${\mathcal A}(X)$
in the spirit of Makar-Limanov and Yu \cite{MLY} and shall assume that $\Delta$
is well ordered relative to the opposite ordering, i.e., any nonempty subset of $\Delta$
has a largest element. Again, if $0\not=\tau\in {\mathcal A}(X)$, we denote by 
$\nu(\tau)$ its leading monomial $\alpha_1u_1$, $\alpha_1\in K$, $u_1\in F(X)$. 
The following Lemma on radicals
of Bergman \cite{B2, B3} plays a crucial role in \cite{MLY}.

\begin{lemma}\label{lemma on radicals}
If $0\not=\tau\in{\mathcal A}(X)$ and $\nu(\tau)=(\beta u)^n$, $\beta\in K$,
$u\in F(X)$, is an $n$-th root, then there exists a $\rho\in {\mathcal A}(X)$ such that
$\tau=\rho^n$.
\end{lemma}

Now we shall show that the polynomial $g$ from the conuterexample to
Conjecture \ref{degree of commutator} serves as a counterexample also to 
Conjecture \ref{conjecture on radicals}.

\begin{theorem}\label{the counterexample to second conjecture}
Let $X=\{x,y\}$, $k\geq 2$, and let 
\[
u=(xy)^kx,\quad v=xy,\quad w=yx,
\]
\[
g=u^2+s,\quad s=v+w.
\]
Then $g$ generates its centralizer $C(g)$ in $K\langle x,y\rangle$ and
its homogeneous component of maximal degree is a square in $K\langle x,y\rangle$.
If $\rho\in {\mathcal A}(x,y)$ is such that $g=\rho^2$, then
$\rho^3$ has no  monomial of positive degree containing 
a negative power of $x$ or $y$.
\end{theorem}

\begin{proof}
We may assume that
\[
\rho=g^{1/2}=u+a_1+a_2+\cdots,
\]
where $a_i$ are homogeneous polynomials such that 
\[
2k+1=\text{deg}(u)>\text{deg}(a_1)>\text{deg}(a_2)>\cdots.
\]
These polynomials are determined step-by-step in a unique way from the condition
\[
g=u^2+s=\rho^2=u^2+(ua_1+a_1u)+(a_1^2+ua_2+a_2u)+\cdots.
\]
Comparing the homogeneous components of $g$ and $\rho^2$ and their degrees, we obtain
\[
ua_1+a_1u=s,\quad \text{deg}(a_1)=\text{deg}(s)-\text{deg}(u)=1-2k,
\]
\[
a_1^2+ua_2+a_2u=0,\quad \text{deg}(a_2)=2\text{deg}(a_1)-\text{deg}(u)=1-6k,
\]
\[
\text{deg}(a_i)=\text{deg}(a_1)+\text{deg}(a_{i-1})-\text{deg}(u)=1-2(2i-1)k.
\]
As in the proof of Theorem \ref{the counterexample to first conjecture},
we have $vu=uw$. Hence $wu^{-1}$ satisfies
\[
u(wu^{-1})+(wu^{-1})u=(uw)u^{-1}+w=(vu)u^{-1}+w=v+w=s,
\]
we conclude that $a_1=wu^{-1}$.
Now
\[
\rho^3=u^3+\sum(u^2a_i+ua_iu+a_iu^2)+\sum(ua_ia_j+a_iua_j+a_ia_ju)+\sum a_ia_ja_l,
\]
\[
\text{deg}(u^2a_1)=2(2k+1)+(1-2k)=2k+3,
\]
\[
\text{deg}(u^2a_i)\leq \text{deg}(u^2a_2)=2(2k+1)+(1-6k)=3-2k<0,\quad i\geq 2,
\]
\[
\text{deg}(ua_ia_j)\leq \text{deg}(ua_1^2)=\text{deg}(u^2a_2)<0,\quad i,j\geq 1,
\]
\[
\text{deg}(a_ia_ja_k)<0,
\]
and we obtain that the component of positive degree of $\rho^3$ is
\[
u^3+(u^2a_i+ua_iu+a_iu^2)=u^3+(u^2(wu^{-1})+u(wu^{-1})u+(wu^{-1})u^2)
\]
\[
=u^3+(u(uw)u^{-1}+(uw)(u^{-1}u)+w(u^{-1}u)u)
\]
\[
=u^3+(uv+uw+wu)=u^3+r=f,
\]
where $f=u^3+r$ is the other polynomial from 
Theorem \ref{the counterexample to first conjecture}.
Hence $\rho^3$ does not contain monomials of positive degree 
with negative powers of variables $x$ and $y$. 
\end{proof}

\section {Acknowlegement}
The authors are grateful to Leonid Makar-Limanov
for his helpful comments and suggestions, and especially for his kind permission
to include Example 1.1.

\end{document}